\documentclass[12pt, a4paper]{amsart}

\usepackage[hmargin=30mm, vmargin=25mm, includefoot, twoside]{geometry}
\usepackage[bookmarksopen=true]{hyperref}

\usepackage{amsfonts,amssymb,verbatim}
\usepackage{latexsym}
\usepackage{mathrsfs}
\usepackage{stmaryrd}
\usepackage{xspace}
\usepackage{enumerate, paralist}
\usepackage{graphicx}
\usepackage[all]{xy}
\usepackage{extarrows}

\usepackage[usenames,dvipsnames]{color}

\usepackage{txfonts, pxfonts}

\usepackage{amsthm}
\usepackage{amsmath}

\newtheorem{thm}{Theorem}[section]
 \newtheorem{cor}[thm]{Corollary}
 \newtheorem{lem}[thm]{Lemma}
 \newtheorem{prop}[thm]{Proposition}

\numberwithin{equation}{section}

 \theoremstyle{definition}
  \newtheorem{defn}[thm]{Definition}
  \newtheorem{question}[thm]{Question}

 \theoremstyle{remark}
 \newtheorem{rem}[thm]{Remark}
  \newtheorem{ex}[thm]{Example}

\newtheorem*{claim*}{Claim}

\def\B{\mathfrak B}

\def\H{\mathcal H}
\def\G{\mathcal G}

\def\supp{\mathrm{supp}}

\def\N{\mathbb N}
\def\C{\mathbb C}
\def\R{\mathbb R}

\def\sgn{\mathrm{sgn}}

\def\ch{\mathrm{ch}}
\def\val{\mathrm{val}}
\def\supp{\mathrm{supp}}

\begin{document}

\title{Measured expanders}

\author{Kang Li, J\'{a}n \v{S}pakula, and Jiawen Zhang}

\address[K. Li]{Department of Mathematics, Friedrich-Alexander-University Erlangen–Nuremberg, Cauerstrasse 11, 91058 Erlangen, Germany.}
\email{kang.li@fau.de}

\address[J. \v{S}pakula]{School of Mathematics, University of Southampton, Highfield, SO17 1BJ, United Kingdom.}
\email{jan.spakula@soton.ac.uk}

\address[J. Zhang]{School of Mathematical Sciences, Fudan University, 220 Handan Road, Shanghai, 200433, China.}
\email{jiawenzhang@fudan.edu.cn}

\date{}

\thanks{KL was supported by the Internal KU Leuven BOF project C14/19/088. KL has also received funding from the European Research Council (ERC) under the European Union's Horizon 2020 research and innovation programme (grant agreement no. 677120-INDEX)}
\thanks{J\v{S} was partially supported by Marie Curie FP7-PEOPLE-2013-CIG Coarse Analysis (631945), and EPSRC Standard Grant EP/V002899/1.}
 
\thanks{JZ was supported by NSFC11871342}

\begin{abstract}
By measured graphs we mean graphs endowed with a measure on the set of vertices. In this context, we explore the relations between the appropriate Cheeger constant and Poincar\'{e} inequalities. We prove that the so-called \emph{Cheeger inequality} holds in two cases: when the measure comes from a random walk, or when the measure has a bounded measure ratio. Moreover, we also prove that our measured (asymptotic) expanders are generalised expanders introduced by Tessera. Finally, we present some examples to demonstrate relations and differences between classical expander graphs and the measured ones. 
The current paper is motivated primarily by our previous work on the rigidity problem for Roe algebras. 
\end{abstract}

\date{\today}
\maketitle

\parskip 4pt


\section{Introduction}
Motivated by our analytic work in  \cite{measured_I,dypartI}, we in this paper explore graph expansion in the context of measured graphs (\emph{i.e.}, their vertex sets are endowed with a measure or weights). There are three essential ways of quantifying expansion in a graph: the Cheeger constant, the spectral gap of the associated graph Laplacian, and the optimal constant in the Poincar\'{e} inequality. This paper aims to address their relations to each other in the measured setting. Most importantly, we have generalised the so-called \emph{Cheeger inequality} to \emph{measured} graphs (see \cite{alon1986eigenvalues,alon1985lambda1,dodziuk1984difference,tanner1984explicit} for the classical Cheeger inequality).
  
Although the notion of graph Laplacian is extremely well understood and established in the unweighted setting, in the setting with weights on the vertex sets there are at least two possibilities to define the graph Laplacian:

The most straightforward one is perhaps as follows: if $m$ is a measure on the vertex set $V$ of a graph $\G$, one can define a unitary isomorphism $U_m:\ell^2(V,m)\to\ell^2(V)$, and declare that the graph Laplacian on $(V,m)$ is $U_m^*\Delta U_m$, where $\Delta$ is the standard graph Laplacian on $\G$. However, this approach does not yield the type of Poincar\'{e} inequality that is required for the results in \cite{measured_I}. 

There is another approach to Laplacians on graphs in the case when the associated measure on the vertex set happens to be a stationary measure of a random walk on the graph, or in other words, there exists a `compatible' measure on the set of edges (see Section \ref{ssec:me from random walk} for details). Although this case already appeared in the literature (see \emph{e.g.} \cite{woess2000random}),  we could not find the exact arguments and precise statements needed for \cite{measured_I}. 
Therefore, we present elementary proofs for the desired Cheeger inequality (Theorem \ref{thm:Cheeger iff Poincare}) and $L^p$-Poincar\'{e} inequality (Proposition \ref{prop: lp Poincare}).

While the random walk approach is undoubtedly elegant, it happens that not every measure on a graph is a stationary measure of some random walk. Instead we sidestep the issue of defining a graph Laplacian directly, and we simply ask the spectral gap to be the best constant for which a Poincar\'{e} inequality holds (cf. Definition \ref{defn:spectral-gap-through-Poincare}). To relate the Cheeger constant and our spectral gap, we require that there is a multiplicative bound on how much the measure can change across an edge (we call it \emph{bounded measure ratio} in Definition \ref{defn:bounded-measure-ratio}). Under this extra assumption, we are able to obtain the Cheeger inequality in Theorem \ref{thm:measured expanders}. 


In Section \ref{Generalised expanders}, we show in Theorem~\ref{thm:generalised expanders from random walks} that our measured expanders (more generally measured asymptotic expanders studied in \cite{measured_I,dypartI}) are $L^p$-generalised expanders introduced by Tessera in \cite{MR2649350} for each $p\in [1,\infty)$. This particularly strengthens a result proved in \cite{measured_I} that measured (asymptotic) expanders do not coarsely embed into any $L^p$-space (see Corollary~\ref{cor:non CE}). In the final section, we present several examples in order to clarify the relation and difference between expander graphs and measured expander graphs (see \emph{e.g.} Corollary~\ref{exm1} and Proposition~\ref{exm2}).

We close the introduction by mentioning that measured (asymptotic) expanders arise naturally from ergodic theory. More precisely, we can always associate a sequence of measured approximating graphs to any measurable action on a probability space. The associated sequence of measured approximating graphs is forms a sequence of measured expanders if and only if the measurable action has a spectral gap. In parallel, the associated sequence of measured approximating graphs forms a sequence of measured asymptotic expanders if and only if the measurable action is strongly ergodic. As a large class of measurable actions which have a spectral gap or are strongly ergodic, one can naturally construct numerous examples of measured (asymptotic) expanders (we refer the reader to \cite[Section 6]{dypartI} for details).


\subsection*{Acknowledgement} We would like to thank G\'{a}bor Elek for pointing out Proposition \ref{prop:example 1} and Lemma \ref{lem:example 2} to us.

\subsection*{Convention}Throughout the paper, all metric spaces are \emph{non-empty and discrete}, and without further explanation a graph refers to a \emph{connected undirected} graph.

\section{Preliminaries}\label{sec:prel}

Let $\G=(V,E)$ be a (connected) graph with vertex set $V$ and edge set $E$. We endow $V$ with the \emph{edge-path} metric $d$, defined to be the length (number of edges) in a shortest path connecting given two points. For $A\subseteq V$, the \emph{vertex boundary} $\partial^V A$ of $A$ is defined to be $\{v\in V: d(v,A)=1\}$, and its \emph{edge boundary} $\partial^E A$ is defined to be the set of all edges in $E$ with one endpoint in $A$ and the other one in $V \setminus A$.

We say that two vertices $v,w \in V$ are \emph{adjacent} if there is an edge in $E$ connecting them, denoted by $v \thicksim_E w$ or just $v \thicksim w$. For a vertex $v\in V$, its \emph{valency} $\val(v)$ is defined to be the number of vertices adjacent to $v$. We say that a graph $\G=(V,E)$ is \emph{$k$-regular} for some $k\in \N$ if $\val(v)=k$ for each $v\in V$, and $\G$ has \emph{bounded valency} if there exists some $K \geq 0$ such that each vertex has valency at most $K$. A sequence of graphs $\{\G_n\}_{n\in \N}$ is said to have \emph{uniformly bounded valency} if there exists some $K \geq 0$ such that each vertex in $\G_n$ has valency bounded by $K$ for all $n$.

For a finite graph $(V,E)$, its \emph{Cheeger constant} is defined to be
\[
\ch(V,E):=\min\big\{\frac{|\partial^V A|}{|A|}: A\subseteq V,\ 0<|A|\leq \frac{|V|}{2}\big\}.
\]
We note that the Cheeger constant $\ch(V,E)$ is positive if and only if $(V,E)$ is a connected graph. 

\begin{defn}\label{defn:expanders}
A sequence of finite (connected)\footnote{We note that the expanding condition already implies that all the graphs $(V_n,E_n)$ in an expander sequence are connected.} graphs $\{(V_n,E_n)\}_{n\in \N}$ is called a sequence of \emph{expanders} if they have uniformly bounded valency, $|V_n| \to \infty$, and $\inf_{n\in \N} \ch(V_n,E_n)>0$.
\end{defn}

In this paper, we study a measured version of expanders as introduced in \cite{measured_I}. Recall that a \emph{finite measured graph} $(V,E,m)$ is a finite (connected and undirected) graph $(V,E)$ equipped with a non-trivial finite measure $m$ defined on the $\sigma$-algebra of all subsets of $V$. We say that $(V,E,m)$ has full support if the support of $m$ is $V$. For a finite measured graph $(V,E,m)$, we define its \emph{Cheeger constant} to be:
\begin{equation}\label{EQ:measured cheeger}
\ch(V,E,m):=\min\big\{ \frac{m(\partial^V A)}{m(A)}:A\subseteq V,\  0<m(A) \leq \frac{1}{2}m(V) \big\}.
\end{equation}
Again, we note that the Cheeger constant of $(V,E,m)$ is positive \emph{if and only if} the full subgraph of $(V,E)$ with $\supp(m)$ as the vertex set is connected.

We recall the following measured version of expanders:

\begin{defn}[{\cite[Definition 4.18]{measured_I}}]\label{defn:measured expanders}
A sequence of \emph{measured expanders} is a sequence of finite measured graphs $\{(V_n,E_n,m_n)\}_{n\in \N}$ such that $\inf_{n\in \N} \ch(V_n,E_n,m_n) >0$.
\end{defn}

\begin{rem}\label{rem:support of m.e.}
For a sequence of finite measured graphs $\{(V_n,E_n,m_n)\}_{n\in \N}$, we consider the sequence of measured subgraphs $\{(\G'_n,m_n)\}_{n\in \N}$ such that $\G'_n$ is the full subgraph in $(V_n,E_n)$ with the vertex set $\supp (m_n)$. Then $\{(V_n,E_n,m_n)\}_{n\in \N}$ is a sequence of measured expanders \emph{if and only if} $\{(\G'_n,m_n)\}_{n\in \N}$ is. The argument is straightforward, because that for any $n\in \N$ and any $A \subseteq V_n$ we have $\partial^{\supp (m_n)}(A \cap \supp (m_n))=\partial^{V_n}(A \cap \supp (m_n)) \cap \supp (m_n) \subseteq \partial^{V_n} A$. 
Hence, without loss of generality we only need to consider measured expanders with full support. 

\end{rem}

Recall from \cite{measured_I} that a sequence of finite measured graphs $\{(V_n,E_n,m_n)\}_{n\in \N}$ is called \emph{ghostly} if
\begin{align}\label{ghostly}
\lim_{n\to \infty} \sup_{v\in V_n} \frac{m_n(v)}{m_n(V_n)}=0.
\end{align}
Note that the condition (\ref{ghostly}) necessarily implies that $|\supp (m_n)|\to \infty$ as $n\to \infty$. However, it is clear that the converse may fail in general. Moreover, if each $m_n$ is the counting measure on $V_n$ then $\{(V_n,E_n,m_n)\}_{n\in \N}$ is ghostly \emph{if and only if} $|V_n|\to \infty$ as $n\to \infty$.

\section{Spectral gaps and Cheeger Inequalities}
It is a classical result that expander graphs can be characterised in terms of the spectral gap of graph Laplacians via Cheeger inequalities \cite{alon1986eigenvalues,alon1985lambda1,dodziuk1984difference,tanner1984explicit}. In this section, we first generalise it to the case of measured expanders coming from random walks, and then move on to a more general case of measured expanders with bounded measure ratio.

\subsection{A review of the Poincar\'{e} inequality for expander graphs}\label{sec:expander graphs}
Let $\G=(V,E)$ be a finite $k$-regular graph for some $k\in \N$, and $\ell^2(V)$ be the Hilbert space of all complex-valued functions on $V$ equipped with the inner product 
\[
\langle f,g \rangle:= \sum_{v\in V} f(v) \overline{g(v)}
\]
for $f,g \in \ell^2(V)$. The graph Laplacian $\Delta$ is a linear operator in $\B(\ell^2(V))$ defined by:
\begin{equation*}
(\Delta f)(v)=f(v)-\frac{1}{k}\sum_{w\in V: w\thicksim v}f(w)
\end{equation*}
for $f\in \ell^2(V)$. It follows from direct calculations that $\Delta$ is a positive operator with norm at most $2$, the constant functions on $V$ are eigenvectors for eigenvalue $0$, and if the graph is connected then $\Delta f=0$ \emph{if and only if} $f$ is constant. The \emph{spectral gap} of the graph $(V,E)$ is defined to be the smallest positive eigenvalue of $\Delta$. Then we have the following Cheeger inequalities:

\begin{prop}[\cite{alon1986eigenvalues,alon1985lambda1,dodziuk1984difference,tanner1984explicit}]\label{prop:classical Poincare}
Let $\G=(V,E)$ be a finite connected $k$-regular graph with positive Cheeger constant $c$ and spectral gap $\lambda$. Then:
\[
\frac{c^2}{2} \leq \lambda \leq 2c.
\]
In particular, we have the following Poincar\'{e} inequality:
\begin{equation}\label{EQ:Poincare for expanders}
\sum_{v,w\in V: v\thicksim w} |f(v)-f(w)|^2\geq c^2 \sum_{v\in V} |f(v)|^2
\end{equation}
for $f\in \ell^2(V)$ with $\sum_{v\in V} f(v)=0$.
\end{prop}

In the rest of this section, we will develop analogous results to measured expanders either from random walks or with bounded measure ratio.

\subsection{Measured expanders from random walks}\label{ssec:me from random walk}
First let us recall some notions from the theory of random walks. For more details, see textbooks \cite{BdlHV:2008, woess2000random}.

A \emph{random walk} or a \emph{Markov chain} on a non-empty set $V$ is a map $r: V \times V \longrightarrow [0,\infty)$ such that $\sum_{v \in V} r(u,v) = 1$ for any $u\in V$. A \emph{stationary measure} $\mu$ for a random walk $r$ is a function $\mu: V \longrightarrow (0,\infty)$ such that $\mu(u)r(u,v)=\mu(v)r(v,u)$ for any $u,v\in V$. A random walk is called \emph{reversible} if it admits at least one stationary measure. In the reversible case, the function $a: V \times V \longrightarrow [0,\infty)$ defined by $a(u,v):=\mu(u)r(u,v)$ is called the \emph{conductance} function. Clearly, $a$ is \emph{symmetric} in the sense that $a(u,v)=a(v,u)$ for all $u,v\in V$, and we also have $\mu(u)=\sum_{v\in V}a(u,v)$ for all $u\in V$. Conversely, let $a: V\times V \to [0,\infty)$ be a symmetric map such that $\mu(u):=\sum_{v\in V} a(u,v)$ is positive and finite for each $u \in V$. Then the formula $r(u,v):=\frac{a(u,v)}{\mu(u)}$ defines a reversible random walk on $V$ with stationary measure $\mu$.

Given a reversible random walk $r$ on a non-empty set $V$ with a stationary measure $\mu$, we can endow $V$ with a (not necessarily connected,\footnote{The constructed graph $(V,E)$ is connected if and only if the random walk $r$ is irreducible (see \cite[Example~5.1.1]{BdlHV:2008} for details).} but undirected) graph $(V,E)$ without multiple edges by requiring that $u\thicksim_E v$ is an edge if and only if $r(u,v)>0$ (which is equivalent to that $r(v,u)>0$). Since the corresponding conductance function $a$ is symmetric, we define $a(e):=a(u,v)=a(v,u)$ for each edge $e\in E$ connecting vertices $u$ and $v$.
For $D \subseteq E$, its \emph{area} is defined to be $a(D):= \sum_{e\in D} a(e)$. We set $a(\emptyset)=0$. 
Now assume that $V$ is finite and let $m$ be a non-trivial finite measure on $V$ of full support. Then we define the \emph{$(\mu,a,m)$-Cheeger constant} of $(V,E,m)$ to be
\begin{align}\label{cheeger random}
\min\big\{\frac{a(\partial^E A)}{\mu(A)}: A\subseteq V \mbox{~with~} 0<m(A) \leq \frac{1}{2}m(V)\big\},
\end{align}
where $\partial^E A$ is the edge boundary of $A$. Similarly as in the remark following (\ref{EQ:measured cheeger}), if the $(\mu,a,m)$-Cheeger constant is positive, then the graph $(V,E)$ is automatically connected.

Given the above data, we consider the following Hilbert space:
$$\ell^2(V;\mu):=\big\{f: V \to \C ~\big|~ \sum_{v\in V} |f(v)|^2\mu(v)< \infty\big\} \quad \mbox{with} \quad \langle f_1, f_2\rangle_\mu:=\sum_{v\in V} f_1(v)\overline{f_2(v)} \mu(v).$$
The \emph{graph Laplacian} $\Delta \in \B(\ell^2(V;\mu))$ associated to the reversible random walk $r$ is defined as
\begin{equation}\label{EQ10}
(\Delta f)(v):=f(v)-\sum_{u\in V: u\thicksim v}f(u)r(v,u)
\end{equation}
for $f\in \ell^2(V;\mu)$ and $v,u \in V$. In fact, the graph Laplacian $\Delta$ is a positive bounded operator with
norm at most $2$. When the constructed graph $(V,E)$ is connected, $\Delta f=0$ if and only if $f$ is constant (see, \emph{e.g.}, \cite[Proposition~5.2.2]{BdlHV:2008} for details).

The \emph{spectral gap} $\lambda$ of $(r,\mu)$ is defined to be the smallest positive eigenvalue of $\Delta$ as defined in (\ref{EQ10}). We now explore a measured version of Proposition \ref{prop:classical Poincare}. We start with a lower bound of the spectral gap, which is stated (without proof) in \cite[Proposition 5.1]{measured_I}.

\begin{prop}\label{prop: Cheeger to Poincare}
Let $r$ be a reversible random walk on a non-empty finite set $V$ with a stationary measure $\mu$ such that $a$ is the associated conductance function and $(V,E)$ is the associated graph structure. If $m$ is a non-trivial finite measure on $V$ of full support such that the $(\mu,a,m)$-Cheeger constant $c$ is positive, then the spectral gap $\lambda$ of $(r,\mu)$ is bounded from below by $\frac{c^2}{2}$. Consequently, we have the following Poincar\'{e} inequality:
\begin{equation}\label{EQ:Poincare from random walk}
\sum_{v,w\in V: v\thicksim_E w}|f(v)-f(w)|^2a(v,w) \geq c^2 \sum_{v\in V} |f(v)|^2\mu(v).
\end{equation}
for any $f\in \ell^2(V;\mu)$ with $\sum_{v\in V}f(v)\mu(v)=0$.
\end{prop}

\begin{rem}
Note that the spectral gap of $(r,\mu)$ does not depend on the auxiliary measure $m$, while the $(\mu,a,m)$-Cheeger constant does. One way to interpret the conclusion of Proposition \ref{prop: Cheeger to Poincare} is that different choices of measures $m$ provide different lower bounds for the (same) spectral gap. 

The possibility of choosing different auxiliary measures $m$ becomes important in the next subsection, where we address the case of measured expanders which do not necessarily arise from random walks.
\end{rem}

\begin{rem}
	Note that the left hand side of (\ref{EQ:Poincare from random walk}) coincides with $2\langle \Delta f,f\rangle$ by a direct calculation. After we arbitrarily choose an orientation on $E$, we rewrite the left hand side as follows: 
	\[
	\sum_{v,w\in V: v\thicksim_E w}|f(v)-f(w)|^2a(v,w) = 2\sum_{e \in E} |f(e^+)-f(e^-)|^2 a(e).
	\]
	In the following, we will use both forms without further explanation.
\end{rem}

The proof of Proposition \ref{prop: Cheeger to Poincare} is similar to the classical case of expander graphs. As we could not find an explicit proof for our setup in the literature, we provide a detailed proof here, following the argument given in \cite[Theorem 1.2.3]{davidoff2003elementary}. We also refer the reader to \cite[Chapter 4.A]{woess2000random} for the case of random walks on infinite graphs.

\begin{proof}[Proof of Proposition \ref{prop: Cheeger to Poincare}]

We fix an orientation on the edge set $E$, \emph{i.e.}, for each edge $e\in E$ we make a choice of orientation. Since the graph $(V,E)$ does not have multiple edges, we may regard $E$ as a subset in $V \times V$ and write $e=(e^-,e^+)$ for $e\in E$. Consider the following Hilbert space:
$$\ell^2(E;a):=\big\{F: E \to \C ~\big|~ \sum_{e \in E} |F(e)|^2a(e) <\infty\big\} \quad \mbox{with}
\quad \langle F,G\rangle:=\sum_{e \in E} F(e)\overline{G(e)}a(e).$$
It is straightforward to check that $\Delta=d^*d$, where
$$d: \ell^2(V;\mu) \rightarrow \ell^2(E;a), \quad (df)(e)=f(e^+)-f(e^-).$$
Let $f: V \to [0,\infty)$ be a non-negative function on $V$, and we consider the following value:
\begin{equation}\label{EQ1}
B_f:=\sum_{e \in E} |f(e^+)^2-f(e^-)^2|\cdot a(e).
\end{equation}
Denote by $\beta_r > \beta_{r-1} > \ldots > \beta_0$ the values $f$ takes, and
set $L_i:=\{v\in V : f(v) \geq \beta_i\}$.

\emph{Step I.} We claim:
$$B_f = \sum_{i=1}^r a(\partial^E L_i) \cdot (\beta_i^2 - \beta_{i-1}^2).$$
Indeed, for any $e=(e^-,e^+) \in E$ we will consider all the level sets $L_i$ it crosses as follows. Suppose $\max\{f(e^-),f(e^+)\}=\beta_{i(e)}$ and $\min\{f(e^-),f(e^+)\}=\beta_{j(e)}$. Replacing each item $|f(e^-)^2-f(e^+)^2|\cdot a(e)$ in (\ref{EQ1}) by
\[
\sum_{k=j(e)+1}^{i(e)} (\beta_k^2-\beta_{k-1}^2)\cdot a(e),
\]
we obtain that
\begin{equation}\label{EQ9}
B_f= \sum_{e \in E} \sum_{k=j(e)+1}^{i(e)} (\beta_k^2-\beta_{k-1}^2)\cdot a(e).
\end{equation}
Observe that every occurrence of the term $\beta_k^2-\beta_{k-1}^2$ in (\ref{EQ9}) corresponds to an edge $e\in E$ which has one endpoint in $L_k$ and the other one outside of $L_k$. Hence, the total number of times that $\beta_k^2-\beta_{k-1}^2$ appears in (\ref{EQ9}) is exactly
$$\sum_{e\in \partial^E L_k} a(e) = a(\partial^E L_k).$$
This proves the claim.

\emph{Step II.} Using the Cauchy--Schwartz inequality in $\ell^2(E;a)$, we have that
\begin{align*}
B_f &= \sum_{e\in E} |f(e^+)+f(e^-)| \cdot |f(e^+)-f(e^-)| \cdot a(e)\\
&\leq  \big( \sum_{e\in E} (f(e^+)+f(e^-))^2\cdot a(e) \big)^{\frac{1}{2}} \cdot \big( \sum_{e\in E} |f(e^+)-f(e^-)|^2\cdot a(e) \big)^\frac{1}{2}\\
&\leq  \big( \sum_{e\in E} (2f(e^+)^2+2f(e^-)^2)\cdot a(e) \big)^{\frac{1}{2}} \cdot \|df\| \\
& = \big( 2\sum_{v\in V} f(v)^2\cdot \sum_{w\sim v}a(v,w) \big)^{\frac{1}{2}} \cdot \|df\|\\
&= \sqrt{2} \cdot \|f\|\cdot \|df\|.
\end{align*}

\emph{Step III.} Now we assume further that $m(\supp f) \leq \frac{1}{2} m(V)$. We claim that $B_f \geq c \cdot \|f\|^2$. Indeed, note that $\beta_0=0$ and that for all $i\geq 1$ we have $m(L_i) \leq \frac{1}{2}m(V)$, so that $a(\partial^E L_i) \geq c \mu(L_i)$. Hence, by the claim in Step I we see that
\begin{align*}
B_f &\geq  \sum_{i=1}^r c\cdot \mu(L_i) \cdot (\beta_i^2 - \beta_{i-1}^2)\\
&= c\cdot \big[ \mu(L_r)\beta_r^2 + \mu(L_{r-1} \setminus L_r)\beta_{r-1}^2 + \ldots +  \mu(L_1 \setminus L_2)\beta_1^2\big]\\
&= c\cdot \|f\|^2.
\end{align*}

\emph{Step IV.} Let $g$ be a real-valued eigenvector of $\Delta$ associated to its smallest positive eigenvalue, \emph{i.e.}, to the spectral gap $\lambda$. Set $V^+:=\{v\in V: g(v)>0\}$ and
$f=\max\{g,0\}$. Since $\sum_{v\in V}g(v)\mu(v) =0$ and $g\neq 0$, by replacing $g$ by $-g$ if necessary, we can assume that
$0<m(V^+) \leq \frac{1}{2} m(V)$. For $v\in V^+$, we have
\begin{align*}
(\Delta f)(v) &= f(v)- \sum_{w\in V: w\thicksim v} f(w) r(v,w)\\ 
&= f(v)- \sum_{w\in V^+: w\thicksim v} g(w) r(v,w) \\
& \leq  g(v)- \sum_{w\in V: w\thicksim v} g(w) r(v,w)\\
& = (\Delta g)(v) = \lambda g(v) = \lambda f(v).
\end{align*}
It follows that
$$\|df\|^2= \langle \Delta f, f\rangle = \sum_{v\in V^+} (\Delta f)(v) \cdot f(v)\cdot \mu(v) \leq \lambda \sum_{v\in V^+}f(v)^2\cdot \mu(v) = \lambda \|f\|^2.$$
Combining all together, we obtain
$$c\cdot \|f\|^2 \leq B_f \leq  \sqrt{2} \cdot \|f\|\cdot \|df\| \leq \sqrt{2\lambda} \cdot \|f\|\cdot \|f\|.$$
Thus, we deduce that $\lambda \geq c^2/2$.

Finally, the Poincar\'e inequality is simply an expanded version of the inequality $\langle \Delta f,f\rangle \geq \frac{c^2}2\langle f,f\rangle$, which holds for all vectors $f$ orthogonal to the zero eigenspace of $\Delta$. Since $c>0$, the only eigenvectors for $0$ are the constant functions, and thus the vectors orthogonal to the zero eigenspace of $\Delta$ are precisely those $f$ satisfying $\sum_{v\in V}f(v)\mu(v)=0$. This concludes the proof.
\end{proof}

Now we consider an upper bound on the spectral gap in the case of $m=\mu$.

\begin{prop}\label{prop:Poincare to Cheeger}
Let $r$ be a reversible random walk on a non-empty finite set $V$ with a stationary measure $\mu$ such that $a$ is the associated conductance function and $(V,E)$ is the associated graph structure. Assume also that the $(\mu,a,\mu)$-Cheeger constant $c$ is positive. Then the spectral gap $\lambda$ of $(r,\mu)$ is bounded from above by $2c$.
\end{prop}

We will follow the proof in \cite{alon1985lambda1}, and start with the following lemma:

\begin{lem}\label{lem:Poincare to Cheeger}
Let $r$ be a reversible random walk on a non-empty finite set $V$ with a stationary measure $\mu$ such that $a$ is the associated conductance function, $(V,E)$ is the associated graph structure and $\lambda$ is the spectral gap of $(r,\mu)$. Assume also that $(V,E)$ is connected. For any subsets $A,B \subseteq V$ with $A \cap B = \emptyset$, we denote $\rho:=d(A,B)$ and $E_A$ (respectively $E_B$) the set of edges with both endpoints in $A$ (respectively in $B$). Then
\[
\lambda \cdot \rho^2 \leq \big(\frac{1}{\mu(A)} + \frac{1}{\mu(B)}\big) \cdot (a(E) - a(E_A) - a(E_B)).
\]
\end{lem}

\begin{proof}
As in the proof of Proposition \ref{prop: Cheeger to Poincare}, we choose an orientation on $E$. Consider the following function on $V$:
\[
g(v):=\frac{\mu(V)}{\mu(A)} - \frac{1}{\rho}\cdot \big( \frac{\mu(V)}{\mu(A)} + \frac{\mu(V)}{\mu(B)} \big) \cdot \min\{d(v,A) ,\rho\}.
\]
It is clear that for any edge $e=(e^-,e^+)$, we have
\[
|g(e^+) - g(e^-)| \leq \frac{1}{\rho}\cdot \big( \frac{\mu(V)}{\mu(A)} + \frac{\mu(V)}{\mu(B)} \big).
\]
Set $\alpha:= \frac{1}{\mu(V)}\sum_{v\in V} g(v)\mu(v)$ and $f:=g-\alpha$. Then
\begin{align*}
\sum_{e\in E}|f(e^+)-f(e^-)|^2a(e) &= \sum_{e\in E}|g(e^+)-g(e^-)|^2a(e) \\
&= \sum_{e\in E \setminus (E_A \cup E_B)}|g(e^+)-g(e^-)|^2a(e)\\ 
&\leq  \frac{1}{\rho^2}\big( \frac{\mu(V)}{\mu(A)} + \frac{\mu(V)}{\mu(B)} \big)^2 (a(E) - a(E_A) - a(E_B)).
\end{align*}
On the other hand, we also have
\begin{align*}
\sum_{v\in V} |f(v)|^2 \mu(v) & \geq  \sum_{v\in A \cup B} |f(v)|^2 \mu(v)\\ 
&= \big( \frac{\mu(V)}{\mu(A)} - \alpha \big)^2 \cdot \mu(A) + \big( \frac{\mu(V)}{\mu(B)} + \alpha \big)^2 \cdot \mu(B) \\
&\geq  \mu(V)^2 \cdot \big( \frac{1}{\mu(A)} + \frac{1}{\mu(B)} \big).
\end{align*}
Hence, it follows from the definition of the spectral gap $\lambda$ that
\[
\frac{1}{\rho^2}\big( \frac{\mu(V)}{\mu(A)} + \frac{\mu(V)}{\mu(B)} \big)^2 (a(E) - a(E_A) - a(E_B)) \geq \lambda \mu(V)^2 \cdot \big( \frac{1}{\mu(A)} + \frac{1}{\mu(B)} \big),
\]
which concludes the proof.
\end{proof}

\begin{proof}[Proof of Proposition \ref{prop:Poincare to Cheeger}]
Given $A \subseteq V$ with $0<\mu(A) \leq \frac{\mu(V)}{2}$, we take $B = V \setminus A$. Since the Cheeger constant $c$ is positive, we know that the graph $(V,E)$ is connected so that $\rho:=d(A,B) = 1$. Applying Lemma \ref{lem:Poincare to Cheeger}, we obtain the following
\[
\lambda \leq \big(\frac{1}{\mu(A)} + \frac{1}{\mu(B)}\big) \cdot a(\partial^E A) \leq \frac{2}{\mu(A)}\cdot a(\partial^E A).
\]
Hence, we conclude that
\[
a(\partial^E A) \geq \frac{\lambda}{2} \mu(A),
\]
as desired.
\end{proof}

Combining Proposition \ref{prop: Cheeger to Poincare} and Proposition \ref{prop:Poincare to Cheeger}, we obtain the following desired result, which is the measured version of Proposition \ref{prop:classical Poincare}:

\begin{thm}\label{thm:Cheeger iff Poincare}
Let $r$ be a reversible random walk on a non-empty finite set $V$ with a stationary measure $\mu$ such that $a$ is the associated conductance function. Assume that the $(\mu,a,\mu)$-Cheeger constant $c$ is positive, and denote the spectral gap of $(r,\mu)$ by $\lambda$. Then
\[
\frac{c^2}{2} \leq \lambda \leq 2c.
\]
\end{thm}

In the following, we prove an $L^p$-version of the Poincar\'{e} inequality arising from random walks. To start, we present an alternative form of the Poincar\'{e} inequality in (\ref{EQ:Poincare from random walk}): For any $f: V \to \C$ with $\sum_{u\in V} f(u)\mu(u)=0$, we see that
\begin{align*}
  \sum_{u,v\in V} &|f(u)-f(v)|^2 \cdot \frac{\mu(u)\mu(v)}{\mu(V)}
  = 2\sum_{u\in V} |f(u)|^2 \big( \sum_{v\in V} \frac{\mu(u)\mu(v)}{\mu(V)} \big) - 2\sum_{u,v\in V} \mathrm{Re}(f(u)\overline{f(v)})\cdot \frac{\mu(u)\mu(v)}{\mu(V)}\\
  &=2\sum_{u\in V} |f(u)|^2\mu(u) - \frac{2}{\mu(V)} \cdot \mathrm{Re}~\big[ \big(\sum_{u\in V}f(u)\mu(u)\big) \cdot \big( \overline{\sum_{v\in V} f(v) \mu(v)}\big)\big] = 2\sum_{u\in V} |f(u)|^2\mu(u).
\end{align*}
In other words, (\ref{EQ:Poincare from random walk}) is equivalent to the following inequality:
\begin{equation}\label{EQ4}
\sum_{u,v\in V: u\thicksim v}|f(u)-f(v)|^2a(u,v) \geq \frac{c^2}{2} \sum_{u,v\in V} |f(u)-f(v)|^2 \frac{\mu(u)\mu(v)}{\mu(V)}.
\end{equation}
Note that both sides of (\ref{EQ4}) are invariant under replacing $f$ by $f+c$ for any $c\in\C$. In particular, (\ref{EQ4}) also holds for all functions $f:V\to\C$ without the restriction that $\sum_{v\in V}f(v)\mu(v)=0$. Consequently, for any $f: V \to \ell^2(\N)$ we have
\begin{equation}\label{EQ5}
\sum_{u,v\in V: u\thicksim v}\|f(u)-f(v)\|^2a(u,v) \geq \frac{c^2}{2} \sum_{u,v\in V} \|f(u)-f(v)\|^2 \frac{\mu(u)\mu(v)}{\mu(V)},
\end{equation}
simply by summing the individual coordinates on both sides of (\ref{EQ4}).

Now we move on to an $L^p$-version of Proposition \ref{prop: Cheeger to Poincare}. This was originally shown in \cite{Mat97} for the case of expander graphs and in \cite[Theorem 3.11]{de2017banach} for the case of random walks. Since we need to deal with the auxiliary measure $m$, we provide a proof here for completeness.

\begin{prop}\label{prop: lp Poincare}
Let $r$ be a reversible random walk on a non-empty finite set $V$ with a stationary measure $\mu$ such that $a$ is the associated conductance function and $(V,E)$ is the associated graph structure. Assume that $m$ is a non-trivial finite measure on $V$ of full support such that the $(\mu,a,m)$-Cheeger constant $c$ is positive. Then for any $p\in [1,\infty)$, there exists a positive constant $c_p$ depending only on $c$ and $p$ such that for any $f:V \to \C$, we have the following $L^p$-Poincar\'{e} inequality:
\begin{equation}\label{EQ:Poincare from random walk lp}
\sum_{u,v\in V: u\thicksim_E v}|f(u)-f(v)|^p a(u,v) \geq c_p \cdot \sum_{u,v\in V} |f(u)-f(v)|^p \frac{\mu(u)\mu(v)}{\mu(V)}.
\end{equation}
\end{prop}

\begin{proof}
First of all, we note that for any $p\in [1,\infty)$ and $a,b\in \R$, we have
$$2^{-\frac{p}{2}}(|a|^p + |b|^p)\leq(\sqrt{|a|^2 + |b|^2})^p \leq 2^{\frac{p}{2}}(|a|^p + |b|^p).$$
Hence, it suffices to consider real functions $f: V \to \R$. We divide the proof into the following two cases.

\emph{Case I: $1\leq p < 2$.} Since there exists an isometric embedding
$$\varphi: (\R,|x-y|^{\frac{p}{2}}) \hookrightarrow \ell^2(\N;\C)$$
(see for example \cite{schoenberg1938metric} and \cite[Example 3.5 in Chapter 1]{wells2012embeddings}), we consider the function $\varphi \circ f: V \to \ell^2(\N;\C)$. From (\ref{EQ5}), we have
\begin{align*}
&\sum_{u,v\in V: u\thicksim v} |f(u)-f(v)|^p a(u,v) = \sum_{u,v\in V: u\thicksim v}\|\varphi \circ f(u) - \varphi \circ f(v)\|^2 a(u,v) \\
&\geq  \frac{c^2}{2} \cdot \sum_{u,v\in V} \|\varphi \circ f(u) - \varphi \circ f(v)\|^2 \frac{\mu(u)\mu(v)}{\mu(V)}= \frac{c^2}{2} \cdot \sum_{u,v\in V} |f(u)-f(v)|^p \frac{\mu(u)\mu(v)}{\mu(V)}.
\end{align*}
Hence, we obtain (\ref{EQ:Poincare from random walk lp}) by setting $c_p:=\frac{c^2}{2}$ for $p\in [1,2)$.

\emph{Case II: $p\geq2$.} Replacing $f$ by $f-\sum_{v\in V} |f(v)|^{\frac{p}{2}} \sgn (f(v)) \mu(v)$, we can assume that $\sum_{v\in V} |f(v)|^{\frac{p}{2}} \sgn (f(v)) \mu(v) =0$. From (\ref{EQ:Poincare from random walk}), we have
\begin{equation}\label{EQ7}
\sum_{u,v\in V: u\thicksim v}\big| |f(u)|^{\frac{p}{2}} \sgn (f(u)) - |f(v)|^{\frac{p}{2}} \sgn (f(v)) \big|^2a(u,v) \geq c^2 \sum_{v\in V} |f(v)|^p \mu(v).
\end{equation}
Note that for any $p>1$ and $a,b\in \R$, we have
$$\big| |a|^p \sgn (a) - |b|^p \sgn (b) \big| \leq p\cdot |a-b|\cdot \big( |a|^{p-1} + |b|^{p-1} \big).$$
So for the left hand side of (\ref{EQ7}), we have
\begin{align*}
L.H.S. & \leq \sum_{u,v\in V: u\thicksim v} \big[ \frac{p}{2} |f(u)-f(v)| \cdot \big( |f(u)|^{\frac{p}{2}-1} + |f(v)|^{\frac{p}{2}-1} \big) \big]^2 \cdot a(u,v)\\
&= \frac{p^2}{4}\cdot \sum_{u,v\in V: u\thicksim v} |f(u)-f(v)|^2 \cdot \big( |f(u)|^{\frac{p}{2}-1} + |f(v)|^{\frac{p}{2}-1} \big)^2 \cdot a(u,v) \\
&\leq  \frac{p^2}{4} \cdot \big( \sum_{u,v\in V: u \thicksim v} |f(u)-f(v)|^p \cdot a(u,v) \big)^{\frac{2}{p}} \cdot \big( \sum_{u,v\in V: u \thicksim v} \big( |f(u)|^{\frac{p}{2}-1} + |f(v)|^{\frac{p}{2}-1} \big)^{\frac{2p}{p-2}} \cdot a(u,v) \big)^{\frac{p-2}{p}}\\
& \leq  \frac{p^2}{4} \cdot \big( \sum_{u,v\in V: u \thicksim v} |f(u)-f(v)|^p \cdot a(u,v) \big)^{\frac{2}{p}} \cdot \big( \sum_{u,v\in V: u \thicksim v} 2^{\frac{2p}{p-2}-1}\big( |f(u)|^{p} + |f(v)|^{p} \big) \cdot a(u,v) \big)^{\frac{p-2}{p}}\\
&\leq \frac{p^2}{4} \cdot 2^{1+\frac{2}{p}} \cdot \big( \sum_{u,v\in V: u \thicksim v} |f(u)-f(v)|^p \cdot a(u,v) \big)^{\frac{2}{p}} \cdot \big( \sum_{v\in V} |f(v)|^p \mu(v)\big)^{1-\frac{2}{p}}.
\end{align*}
Combined with (\ref{EQ7}), we obtain that
\begin{align*}
\sum_{u,v\in V: u \thicksim v}|f(u)-f(v)|^p a(u,v) &\geq  \big(\frac{4c^2}{p^2 \cdot 2^{1+\frac{2}{p}}}\big)^{\frac{p}{2}} \sum_{v\in V} |f(v)|^p \mu(v) \\
&\geq \big(\frac{4c^2}{p^2 \cdot 2^{1+\frac{2}{p}}}\big)^{\frac{p}{2}}\cdot \frac{1}{2^{p+1}} \sum_{u,v\in V} |f(u)-f(v)|^p \frac{\mu(u)\mu(v)}{\mu(V)}.
\end{align*}
Consequently, if $p\geq 2$ then (\ref{EQ:Poincare from random walk lp}) holds for $c_p:=\big(\frac{4c^2}{p^2 \cdot 2^{1+\frac{2}{p}}}\big)^{\frac{p}{2}}\cdot \frac{1}{2^{p+1}}$,
and we have finished the proof.
\end{proof}

In the remainder of this subsection, let us relate the above discussion on random walks to the notion of measured expanders introduced in Definition \ref{defn:measured expanders}. To this end, we need the following easy lemma. 

\begin{lem}\label{lem:me from random walk}
Let $r$ be a reversible random walk on a non-empty finite set $V$ with a stationary measure $\mu$ such that $a$ is the associated conductance function and $(V,E)$ is the associated graph structure. Then the $(\mu,a,\mu)$-Cheeger constant does not exceed the Cheeger constant $\mathrm{ch}(V,E,\mu)$ defined in (\ref{EQ:measured cheeger}). 
\end{lem}

\begin{proof}
Since $v \thicksim u$ if and only if $a(u,v)>0$, it follows that for every $A \subseteq V$ we have 
\[
a(\partial^E A) = \sum_{e\in \partial^E A} a(e) = \sum_{u\in \partial^V A} \sum_{v\in A: v \thicksim u} a(u,v) \leq \sum_{u\in \partial^V A} \sum_{v\in V: v \thicksim u} a(u,v) = \sum_{u\in \partial^V A} \mu(u) = \mu(\partial^V A).
\]
This concludes the proof.
\end{proof}

\begin{rem}
It is plausible that the two notions of Cheeger constants in Lemma~\ref{lem:me from random walk} may be different in general. It is also interesting to find conditions under which they are bound to each other. However, we did not pursue this line of questions.
\end{rem}

Consequently, we obtain the following:

\begin{prop}\label{prop:random walk are me}
Let $\{(V_n,r_n,\mu_n)\}_{n\in \N}$ be a sequence of reversible random walks $r_n$ on non-empty finite sets $V_n$ with stationary measures $\mu_n$. For each $n\in \N$, let $a_n$ be the associated conductance function and $E_n$ the associated edge set on $V_n$. Assume that there exists a constant $c>0$ such that each $(\mu_n,a_n,\mu_n)$-Cheeger constant is greater than $c$. Then $\{(V_n,E_n,\mu_n)\}_{n\in \N}$ is a sequence of measured expanders.
\end{prop}

Let us extract the following definition:

\begin{defn}\label{defn:me from random walks}
A sequence of finite measured graphs $\{(V_n,E_n,\mu_n)\}_{n\in \N}$ is called a sequence of \emph{measured graphs from random walks} if for each $n\in \N$, there exists a reversible random walk $r_n$ on $V_n$ with the stationary measure $\mu_n$ and the associated edge set is $E_n$. Furthermore, it is called a sequence of \emph{measured expanders from random walks} if there exists a constant $c>0$ satisfying that each $(\mu_n,a_n,\mu_n)$-Cheeger constant is greater than $c$, where $a_n$ is the conductance function for $(r_n,\mu_n)$.
\end{defn}

As a consequence of Theorem \ref{thm:Cheeger iff Poincare}, we obtain the following reformulation:

\begin{thm}\label{thm:main thm for random walks}
Let $\{(V_n,E_n,\mu_n)\}_{n\in \N}$ be a sequence of measured graphs from random walks $\{r_n\}_{n\in \N}$, and $\Delta_n \in \B(\ell^2(V_n;\mu_n))$ be the graph Laplacian defined in (\ref{EQ10}) for each $n\in \N$. Then the following are equivalent:
\begin{enumerate}
 \item $\{(V_n,E_n,\mu_n)\}_{n\in \N}$ is a sequence of measured expanders from random walks;
 \item There exists a constant $c>0$ such that the spectral gap of $(r_n,\mu_n)$ is bounded from below by $c$ for each $n\in \N$;
 \item There exists a constant $c'>0$ such that the spectrum of $\Delta_n$ is contained in $\{0\} \cup [c',2]$ for each $n\in \N$.
\end{enumerate}
\end{thm}

\subsection{Measured expanders with bounded measure ratio}\label{ssec:me with bdd ratio}

In this subsection, we will focus on more general measured graphs which do not necessarily come from random walks.

The Cheeger constant (using the vertex expansion) defined in \eqref{EQ:measured cheeger} certainly makes sense for general finite measured graphs. However, the classical Cheeger constant from random walks (see \eqref{cheeger random}) requires a compatible measure $a$ (\emph{i.e.,} the conductance function) on the set of edges in order to control the spectral gap (see Theorem~\ref{thm:Cheeger iff Poincare} for the Cheeger inequality). To extend the Cheeger inequality to a general case, we restrict ourselves to finite measured graphs with bounded measure ratio. Nevertheless, this notion is flexible enough for our purposes in the study of Roe algebras and measured asymptotic expanders in \cite{measured_I}\footnote{It is shown in \cite[Corollary~4.21]{measured_I} that measured asymptotic expanders admit a uniform exhaustion by measured expander graphs with bounded measure ratios.}.

\begin{defn}\label{defn:bounded-measure-ratio}
	Let $(V,E,m)$ be a finite measured graph and $s\in (0,1)$. We say that it has \emph{measure ratio bounded by $s$} if for every edge $u\thicksim_E v$, we have $s\cdot m(v)\leq m(u)\leq \frac{1}{s}\cdot m(v)$. We say that a family of finite measured graphs $\{(V_n,E_n,m_n)\}_{n\in\N}$ has \emph{bounded measure ratio} if there exists $s\in (0,1)$ such that for every $n\in\N$, the finite measured graph $(V_n,E_n,m_n)$ has measure ratio bounded by $s$.
\end{defn}

\begin{rem}\label{rem:full support of me}
\begin{itemize}
\item For a single finite measured graph $(V,E,m)$ such that $m$ has full support, it always has bounded measure ratio by some $s\in (0,1)$. However, we will require a uniform bound on the measure ratio for a family of finite measured graphs (such as measured expanders). 

\item Recall that our standing assumption is that a finite measured graph $(V,E,m)$ is connected and $m$ is non-trivial. It follows directly that if $(V,E,m)$ has measure ratio bounded by some $s\in (0,1)$, then $m$ must have full support.
\end{itemize}
\end{rem}

The following lemma has been worked out in \cite{measured_I} and it builds an auxiliary random walk whose
stationary measure $\mu$ can control the original measure $m$ for the given finite measured graph with bounded measure ratio.   

\begin{lem}[{\cite[Lemma 5.2]{measured_I}}]\label{lem:auxiliary random walk}
Let $(V,E,m)$ be a finite measured graph such that $m$ has full support. Then there exists a (unique) reversible random walk $r: V \times V \to [0,\infty)$ with a (unique) stationary measure $\mu: V \to (0,\infty)$ and the associated conductance function $a:V \times V \to [0,\infty)$ satisfying the following:
\begin{enumerate}
  \item $a(u,v)=m(u)+m(v)$ whenever $u,v\in V$ such that $u \thicksim_{E} v$;
  \item For $u,v\in V$, we have $r(u,v)>0$ if and only if $u\thicksim_{E} v$.
\end{enumerate} 
If we additionally suppose that $(V,E,m)$ has valency bounded by $K\geq 1$, measure ratio bounded by $s\in (0,1)$ and Cheeger constant $c>0$, then
\begin{enumerate}
  \item[(3)] For $u\in V$, we have $\frac{s}{K(1+s)}\mu(u)\leq m(u) \leq \frac{1}{1+s}\mu(u)$;
  \item[(4)] The $(\mu,a,m)$-Cheeger constant in \eqref{cheeger random} is bounded below by $\frac{cs}{K}$.
\end{enumerate} 
\end{lem}

\begin{rem}
Note that when $m$ is the counting measure on $V$, the random walk $r$ constructed in Lemma \ref{lem:auxiliary random walk} is exactly the regular random walk on $V$, \emph{i.e.}, $r(u,v) = \frac{1}{\mathrm{val}(u)}$ for $u \thicksim v$ and $0$ otherwise. In this case, the associated stationary measure $\mu$ coincides with twice of the valency function. If $(V,E)$ has constant valency, we have recovered the setting of expander graphs discussed in Section \ref{sec:expander graphs}.
\end{rem}

Based on the construction of random walks from Lemma \ref{lem:auxiliary random walk}, we propose the following notion of spectral gap for measured graphs as the best constant for which a certain version of Poincar\'{e} inequality holds:

\begin{defn}\label{defn:spectral-gap-through-Poincare}
The \emph{spectral gap} of a finite measured graph $(V,E,m)$ with full support is the supremum of $\lambda$ satisfying the following inequality:
\begin{equation}\label{EQ:spectral gap for measured}
	\sum_{u,v\in V: u\thicksim_E v}|f(u)-f(v)|^2 (m(u) + m(v)) \geq 2\lambda \sum_{u\in V} |f(u)|^2 m(u),
\end{equation}
for any $f\in \ell^2(V;m)$ with $\sum_{v\in V}f(v)m(v)=0$. Equivalently,
\begin{equation}\label{EQ:spectral gap for measured 2}
	\sum_{u,v\in V: u\thicksim_E v}|f(u)-f(v)|^2 (m(u) + m(v)) \geq \lambda \sum_{u,v\in V} |f(u)-f(v)|^2 \cdot \frac{m(u)m(v)}{m(V)}
\end{equation}
for any $f\in \ell^2(V;m)$.
\end{defn}

Note that the equivalence between (\ref{EQ:spectral gap for measured}) and (\ref{EQ:spectral gap for measured 2}) follows exactly from the same calculations used to obtain (\ref{EQ4}). We leave the details to the reader.

The following lemma justifies the name of spectral gap in Definition \ref{defn:spectral-gap-through-Poincare} by the spectral gap of some genuine Laplacian operator:

\begin{lem}\label{lem:Lambda operator}
Let $(V,E,m)$ be a finite measured graph with full support, and $(r,\mu)$ be the random walk on $(V,E)$ from Lemma \ref{lem:auxiliary random walk}. Let $\Delta \in \B(\ell^2(V;\mu))$ be the graph Laplacian defined in (\ref{EQ10}), and $W: \ell^2(V;m) \to \ell^2(V;\mu)$ be the identity operator. Then the spectral gap of $(V,E,m)$ in Definition \ref{defn:spectral-gap-through-Poincare} coincides with the smallest positive eigenvalue of the operator $\Lambda:=W^* \Delta W \in \B(\ell^2(V;m))$.
\end{lem}

\begin{proof}
First note that for $f: V \to \C$, the function $Wf$ is constant if and only if $f$ is. Hence the kernel of $\Lambda$ consists of constant functions on $V$. Also note that for any $f\in \ell^2(V;m)$, we have:
\[
\langle \Lambda f, f \rangle_{m} = \langle \Delta W f, W f \rangle_{\mu} = \frac{1}{2}\sum_{u,v\in V: u\thicksim v}|f(u)-f(v)|^2\big( m(u)+m(v) \big).
\]
Hence the smallest positive eigenvalue of $\Lambda$ is the largest $\lambda>0$ such that 
\[
\langle \Lambda f, f \rangle_m \geq \lambda \sum_{v\in V} |f(v)|^2m(v),
\]
\emph{i.e.},
\[
\frac{1}{2}\sum_{u,v\in V: u\thicksim v}|f(u)-f(v)|^2\big( m(u)+m(v) \big) \geq \lambda \sum_{v\in V} |f(v)|^2m(v),
\]
for all $f\in \ell^2(V;m)$ with $\sum_{v\in V}f(v)m(v)=0$. This coincides with the spectral gap from Definition \ref{defn:spectral-gap-through-Poincare}.
\end{proof}

The next result provides a control on the spectral gap of a finite measured graph by that of the associated random walk. This will allow us to bound or estimate spectral gap in terms of the Cheeger constant for measured graphs.

\begin{lem}\label{lem:controls on spectral gap}
Let $(V,E,m)$ be a finite measured graph with valency bounded by $K$ and measure ratio bounded by $s\in (0,1)$, and let $\lambda$ be the spectral gap of $(V,E,m)$. Let $(r,\mu)$ be the random walk on $(V,E)$ from Lemma \ref{lem:auxiliary random walk}, and $\lambda'$ be the spectral gap of $(r,\mu)$. Then:
\[
\frac{s(1+s)}{K} \cdot \lambda' \leq \lambda \leq \frac{K^2(1+s)}{s^2} \cdot \lambda'.
\] 
\end{lem}

\begin{proof}
By Definition \ref{defn:spectral-gap-through-Poincare}, the number $\lambda$ is the supremum of $\lambda_1$ satisfying 
\begin{equation*}
	\sum_{u,v\in V: u\thicksim v}|f(u)-f(v)|^2 (m(u) + m(v)) \geq \lambda_1 \sum_{u,v\in V} |f(u)-f(v)|^2 \cdot \frac{m(u)m(v)}{m(V)}
\end{equation*}
for any $f: V \to \C$. On the other hand, the number $\lambda'$ is the supremum of $\lambda_2$ satisfying 
\begin{equation*}
	\sum_{u,v\in V: u\thicksim v}|f(u)-f(v)|^2 (m(u) + m(v)) \geq \lambda_2 \sum_{u,v\in V} |f(u)-f(v)|^2 \cdot \frac{\mu(u)\mu(v)}{\mu(V)}
\end{equation*}
for any $f: V \to \C$. Now note from condition (3) in Lemma \ref{lem:auxiliary random walk} that $\frac{s}{K(1+s)}\mu(v)\leq m(v) \leq \frac{1}{1+s}\mu(v)$ for all $v\in V$, hence we obtain:
\[
\frac{s^2}{K^2(1+s)}\cdot \sum_{u,v\in V} |f(u)-f(v)|^2 \cdot \frac{\mu(u)\mu(v)}{\mu(V)} \leq \sum_{u,v\in V} |f(u)-f(v)|^2 \cdot \frac{m(u)m(v)}{m(V)}
\]
and
\[
\sum_{u,v\in V} |f(u)-f(v)|^2 \cdot \frac{m(u)m(v)}{m(V)} \leq \frac{K}{s(1+s)} \cdot \sum_{u,v\in V} |f(u)-f(v)|^2 \cdot \frac{\mu(u)\mu(v)}{\mu(V)}.
\]
This concludes the proof.
\end{proof}

Now we can apply Proposition \ref{prop: Cheeger to Poincare} to the above auxiliary random walk $r$ and obtain the required Poincar\'{e} inequality for measured graphs in this case. Therefore, we reach the following lower bound for the spectral gap, which refines \cite[Corollary 5.3]{measured_I} in the case of a Hilbert space:

\begin{prop}\label{prop:Cheeger to Poincare for measured expanders}
Let $(V,E,m)$ be a finite measured graph with valency bounded by $K$, measure ratio bounded by $s\in (0,1)$ and Cheeger constant $c$. Then the spectral gap of $(V,E,m)$ is bounded from below by $\frac{c^2s^3(1+s)}{2K^3}$.
\end{prop}

\begin{proof}
Applying Proposition \ref{prop: Cheeger to Poincare} to the random walk $r$ with the associated stationary measure $\mu$ from Lemma \ref{lem:auxiliary random walk}, we obtain that the spectral gap of $(r,\mu)$ is bounded from below by $\frac{c^2s^2}{2K^2}$. Now the left inequality in Lemma \ref{lem:controls on spectral gap} concludes the proof.
\end{proof}

At this point, recall that in \cite{measured_I}, we apply Proposition \ref{prop: lp Poincare} directly to the auxiliary measure and obtain an $L^p$-version of Proposition \ref{prop:Cheeger to Poincare for measured expanders} for $p \in [1,\infty)$ as follows:

\begin{prop}[{\cite[Corollary 5.3]{measured_I}}]\label{prop:Poincare for measured expanders lp}
Let $(V,E,m)$ be a finite measured graph with valency bounded by $K$, measure ratio bounded by $s\in (0,1)$ and Cheeger constant $c$. Then for any $p\in [1,\infty)$, there exists a positive constant $c_p$ depending only on $c,s,p,K$ such that for any $f:V \to \C$, we have the following $L^p$-Poincar\'{e} inequality:
\begin{equation}\label{EQ:lp Poincare for measured expanders}
\sum_{u,v\in V: u\thicksim v}|f(u)-f(v)|^p (m(u) + m(v)) \geq c_p \sum_{u,v\in V} |f(u)-f(v)|^p \frac{m(u)m(v)}{m(V)}.
\end{equation}
\end{prop}

Now we address the opposite direction in the Hilbert space case, \emph{i.e.}, an upper bound for the spectral gap. We have the following result analogous to Proposition \ref{prop:Poincare to Cheeger}:

\begin{prop}\label{prop:Poincare to Cheeger for measured expanders}
Let $(V,E,m)$ be a finite measured graph with valency bounded by $K$, measure ratio bounded by $s\in (0,1)$, and spectral gap $\lambda$. Then the Cheeger constant of $(V,E,m)$ is bounded from below by $\frac{s\lambda}{2(1+s)K}$.
\end{prop}

\begin{proof}
Let $A \subseteq V$ with $0<m(A) \leq \frac{m(V)}{2}$. We now apply the argument in the proof of Lemma \ref{lem:Poincare to Cheeger}, with $B=V\setminus A$ (and hence $\rho=1$ and $E\setminus(E_A\cup E_B)=\partial^E(A)$), $a(e)=m(e^{-})+m(e^{+})$, $\mu=m$,\footnote{Note that the assumption in Lemma \ref{lem:Poincare to Cheeger} that $\mu$ is the stationary measure is not used in its proof.} and function $f$ as defined there. This provides two inequalities:
\begin{align*}
	\sum_{e\in E} |f(e^+)-f(e^-)|^2 (m(e^-) + m(e^+)) &\leq  \big( \frac{m(V)}{m(A)} + \frac{m(V)}{m(X \setminus A)} \big)^2 \sum_{e\in \partial^E A} (m(e^-) + m(e^+)),\\
	\sum_{v\in V} |f(v)|^2 m(v) & \geq  m(V)^2 \cdot \big( \frac{1}{m(A)} + \frac{1}{m(X \setminus A)} \big).
\end{align*}
Observe that in our case we also have
\[\sum_{e\in \partial^E A} (m(e^-) + m(e^+))\leq m(\partial^VA)(1+\frac{1}{s})K.
\]
Combining the above three inequalities yields
\begin{equation}\label{eqn:looks-like-opposite-of-Poincare}
\sum_{e\in E} |f(e^+)-f(e^-)|^2 (m(e^-) + m(e^+)) \leq 
\big( \frac{1}{m(A)} + \frac{1}{m(X \setminus A)} \big)(1+\frac{1}{s})Km(\partial^VA)\sum_{v\in V} |f(v)|^2 m(v).
\end{equation}
Hence it follows from the definition of $\lambda$ that\footnote{Note that the first sum in the displayed equation \eqref{eqn:looks-like-opposite-of-Poincare} encounters every edge only once, while the corresponding sum in Definition \ref{defn:spectral-gap-through-Poincare} runs over adjacent pairs of vertices.}
\[
\big( \frac{1}{m(A)} + \frac{1}{m(X \setminus A)} \big) m(\partial^V A)(1+\frac{1}{s})K \geq \lambda,
\]
which implies that
\[
m(\partial^V A) \geq \frac{s\lambda}{2(1+s)K} m(A).
\]
This finishes the proof.
\end{proof}

Finally combining Proposition \ref{prop:Cheeger to Poincare for measured expanders} and Proposition \ref{prop:Poincare to Cheeger for measured expanders}, we obtain the following. Readers may compare with Theorem \ref{thm:Cheeger iff Poincare}. 

\begin{thm}\label{thm:measured expanders}
Let $(V,E,m)$ be a finite measured graph with valency bounded by $K$, measure ratio bounded by $s\in (0,1)$, Cheeger constant $c$ and spectral gap $\lambda$. Then
\[
\frac{c^2s^3(1+s)}{2K^3} \leq \lambda \leq \frac{2(1+s)Kc}{s}.
\]
\end{thm}

Consequently, we apply Theorem \ref{thm:measured expanders} to a sequence of finite measured graphs and obtain the following (readers may compare with Theorem \ref{thm:main thm for random walks}):

\begin{thm}\label{thm:main thm for bounded ratio}
Let $\{(V_n,E_n,m_n)\}_{n\in \N}$ be a sequence of finite measured graphs with uniformly bounded valency and bounded measure ratio. For each $n\in \N$, let $\Lambda_n \in \B(\ell^2(V_n,m_n))$ be the operator from Lemma \ref{lem:Lambda operator}. Then the following are equivalent:
\begin{enumerate}
 \item $\{(V_n,E_n,m_n)\}_{n\in \N}$ is a sequence of measured expanders;
 \item There exists a constant $c>0$ such that the spectral gap of $(V_n,E_n,m_n)$ is bounded from below by $c$ for each $n\in \N$;
 \item There exist constants $c',\kappa>0$ such that the spectrum of $\Lambda_n$ is contained in $\{0\} \cup [c',\kappa]$ for each $n\in \N$.
\end{enumerate}
\end{thm}

\section{Generalised expanders}\label{Generalised expanders}

In this section, we relate measured expanders and their asymptotic version with the notion of generalised expanders. First recall that the notion of generalised expanders was introduced by Tessera in \cite{MR2649350} as a generalisation of expander graphs and as (the sole) obstruction to coarse embeddability into a Hilbert space. We will state in the case of general $L^p$-spaces for $p\in [1,\infty)$.

\begin{defn}[{\cite[Definition 16]{MR2649350}}]\label{defn:generalised expanders}
Let $p\in [1,\infty)$. A sequence of finite metric spaces $\{(X_n,d_n)\}_{n\in \N}$ is called a sequence of \emph{$L^p$-generalised expanders} if for every function $\rho_+: [0,\infty) \to [0,\infty)$, there exists a constant $k>0$ and a sequence $r_n \to \infty$ such that for each $n\in \N$ there exists a probability measure $\mu_n$ on $X_n \times X_n$ satisfying the following:
\begin{enumerate}
 \item $\mu_n(x,y) = \mu_n(y,x)$ for all $x,y\in X_n$;
 \item $\mu_n(x,y)=0$ if $d_n(x,y) \leq r_n$;
 \item for every measure space $(Z,\nu)$ and any map $f: X_n \to L^p(Z,\nu)$ satisfying $\|f(x)-f(y)\|_p \leq \rho_+(d_n(x,y))$ for any $x,y\in X_n$, the following inequality holds:
 \[
 \sum_{x,y\in X_n} \|f(x)-f(y)\|_p^p \mu_n(x,y) \leq k^p.
 \]
\end{enumerate}
\end{defn}

\begin{prop}[{\cite[Corollary 17]{MR2649350}}]\label{prop:generalised expanders}
Let $p\in [1,\infty)$. A metric space $X$ does not coarsely embed into any $L^p$-space \emph{if and only if} it has a coarsely embedded sequence of $L^p$-generalised expanders.
\end{prop}

The main aim of this section is to show that ghostly measured asymptotic expanders, which is an asymptotic version of measured expanders introduced in \cite{measured_I}, are $L^p$-generalised expanders for any $p\in [1,\infty)$. Let us recall some notions. 

\begin{defn}[{\cite[Definition 2.4]{measured_I}}]\label{def:ghostly}
We say that $\{(X_n,d_n,m_n)\}_{n\in \N}$ is a sequence of \emph{finite measured metric spaces} if each $(X_n,d_n)$ is a finite metric space, and each $m_n$ is a non-trivial and finite measure defined on the $\sigma$-algebra of all subsets of $X_n$. Moreover, the sequence $\{(X_n,d_n,m_n)\}_{n\in \N}$ is called \emph{ghostly} if
\begin{align*}
\lim_{n\to \infty} \sup_{x\in X_n} \frac{m_n(x)}{m_n(X_n)}=0,
\end{align*}
and said to have \emph{uniformly bounded geometry} if for all $R>0$, we have 
$$\sup_{n\in \N}\sup_{x\in X_n}|B(x,R)|<\infty.$$
\end{defn}

Note that in the case of finite measured graphs, the above notions are compatible with those defined in Section \ref{sec:prel}. The following notion is the asymptotic version of measured expanders, which was introduced in \cite{dypartI} and extensively studied in \cite{measured_I} (see also \cite{structure,Intro}):

\begin{defn}[{\cite[Definition~6.1]{dypartI}}]\label{def: measured asymptotic expander}
A sequence of finite measured metric spaces $\{(X_n,d_n,m_n)\}_{n \in \N}$ is called a sequence of \emph{measured asymptotic expanders} if for any $\alpha \in (0,\frac{1}{2}]$ there exist $c_\alpha\in (0,1)$ and $R_\alpha>0$ such that for any $n \in \N$ and $A \subseteq X_n$ with $\alpha \cdot m_n(X_n) \leq m_n(A) \leq \frac{1}{2} m_n(X_n)$, we have $m_n(\partial_{R_\alpha} A) > c_\alpha\cdot m_n(A)$.

In this case, we call functions $\underbar{c}:\alpha \mapsto c_\alpha$ and $\underbar{R}:\alpha \mapsto R_\alpha$ from $(0,\frac{1}{2}]$ to $(0,\infty)$ \emph{parameter functions} of $\{(X_n,d_n,m_n)\}_{n \in \N}$, and $\{(X_n,d_n,m_n)\}_{n\in \N}$ is called a sequence of \emph{measured $(\underbar{c},\underbar{R})$-asymptotic expanders}.
\end{defn}

Clearly any measured expander is a measured asymptotic expander. A form of a converse is provided by the structure result from \cite{measured_I}, which we recall below. Roughly speaking, it shows that measured asymptotic expanders are necessarily ``exhausted'' by a sequence of measured expanders. It is also a key ingredient to prove the main result of this section.

\begin{prop}[{\cite[Corollary 4.20]{measured_I}}]\label{prop:structure thm from part I}
Let $\{(X_n,d_n,m_n)\}_{n \in \N}$ be a sequence of finite measured metric spaces with uniformly bounded
geometry. Then the following are equivalent:
\begin{enumerate}
  \item $\{(X_n,d_n,m_{n})\}_{n \in \N}$ is a sequence of measured asymptotic expanders;
  \item there exist $c>0$, a sequence $\{\alpha_k\}_{k \in \N}$ in $(0,1)$ with $\alpha_k\to 0$, a sequence $\{s_k\}_{k\in \N}$ in $(0,1)$, and a positive sequence $\{R_k\}_{k\in \N}$ such that for any $n, k \in \N$ there exist a finite graph $(V_{n,k},E_{n,k})$ and a $R_k$-Lipschitz injective map $i_{n,k}:V_{n,k} \to X_n$ satisfying the following:
\begin{itemize}
   \item[(i)] the pullback measure $m_{n,k}:=i_{n,k}^*(m_n)$ on $V_{n,k}$ has full support and $m_{n,k}(V_{n,k}) \geq (1-\alpha_k)\cdot m_n(X_n)$;
    \item[(ii)] for each $k\in \N$, $\{(V_{n,k},E_{n,k},m_{n,k})\}_{n\in \N}$ is a sequence of measured expanders with uniformly bounded valency and Cheeger constants bounded below by $c$;
    \item[(iii)] for any adjacent vertices $u\thicksim_{E_{n,k}} v$, we have $s_k m_{n,k}(u) \leq m_{n,k}(v) \leq \frac{1}{s_k}m_{n,k}(u)$.
      \end{itemize}
\end{enumerate}
\end{prop}

The following is the main result of this section:
\begin{thm}\label{thm:generalised expanders from random walks}
Let $\{(X_n,d_n,m_n)\}_{n \in \N}$ be a sequence of ghostly measured asymptotic expanders with uniformly bounded geometry.
Then for every $p\in [1,\infty)$, the sequence $\{(X_n,d_n,m_n)\}_{n \in \N}$ is a sequence of $L^p$-generalised expanders.
\end{thm}

We need the following lemma:

\begin{lem}\label{lem:Cheeger to Poincare measured case}
Let $(V,E,m)$ be a finite measured graph with valency bounded by $K$, measure ratio bounded by $s\in (0,1)$ and (positive) Cheeger constant $c$. Let $p\in [1,\infty)$ and $\rho_+: [0,\infty) \to [0,\infty)$ be a function. Then there exists a constant $\kappa$ depending on $K,s,c,p,\rho_+$ such that for any measure space $(Z,\nu)$ and any map $f: V \to L^p(Z,\nu)$ with $\|f(u)-f(v)\|_p \leq \rho_+ (d(u,v))$ for every $u,v\in V$ (where $d$ is the edge-path metric on $V$), we have the following inequality:
\[
\sum_{u,v\in V} \|f(u)-f(v)\|_p^p m(u)m(v) \leq \kappa \cdot m(V)^2.
\]
\end{lem}

\begin{proof}
Applying Proposition \ref{prop:Poincare for measured expanders lp} and the Poincar\'{e} inequality (\ref{EQ:lp Poincare for measured expanders}) to each of the coordinate function $f_x(v):=f(v)(x)$ and integrating these inequalities over $X$, we obtain that there exists a constant $c_p>0$ depending on $K,s,c,p$ such that:
\begin{align*}
&c_p \sum_{u,v\in V} \|f(u)-f(v)\|_p^p \cdot \frac{m(u)m(v)}{m(V)} \leq \sum_{u,v\in V: u\thicksim v}\|f(u)-f(v)\|_p^p (m(u) + m(v)) \\
&\leq \rho_+(1)^p \sum_{u\in V} \sum_{v\in V: v\thicksim u} \Big(m(u) + m(v)\Big) \leq \rho_+(1)^p \sum_{u\in V} K(1+\frac{1}{s})m(u) = \frac{\rho_+(1)^pK(1+s)}{s}m(V).
\end{align*}
Taking $\kappa:= \frac{\rho_+(1)^pK(1+s)}{sc_p}$, we conclude the proof.
\end{proof}

\begin{proof}[Proof of Theorem \ref{thm:generalised expanders from random walks}]
We fix a $p \in [1,\infty)$ and a function $\rho_+:[0,\infty) \to [0,\infty)$. Without loss of generality, we assume that each $m_n$ is a probability measure on $X_n$. Set $\gamma_n:=\sup_{x\in X_n} m_n(x)$, which tends to $0$ as $n\to \infty$. By Proposition \ref{prop:structure thm from part I}, there exist $c>0$, $R>0$, $K>0$ and $s\in (0,1)$ such that for any $n\in \N$ there exist a finite graph $(V_n,E_n)$ and an $R$-Lipschitz injective map $i_n:V_n \to X_n$ satisfying the following:
\begin{itemize}
   \item[(i)] the measure $m'_n:=i_n^*(m_n)$ on $V_n$ has full support and $m'_n(V_n) \geq \frac{1}{2}$;
   \item[(ii)] $\{(V_n,E_n,m'_n)\}_{n\in \N}$ is a sequence of measured expanders with valency bounded by $K$, Cheeger constants above $c$ and measure ratio bounded by $s$.
\end{itemize}

For each $n\in \N$, set $Y_n:=i_n(V_n) \subseteq X_n$, define a measure $\mu_n$ on $X_n \times X_n$ by
\begin{equation*}
\mu_n(x,y):=
\begin{cases}
  ~m_n(x)m_n(y), & \mbox{if~} x,y\in Y_n; \\
  ~0, & \mbox{otherwise}
\end{cases}
\end{equation*}
and let $r_n:=\log_K(\frac{1}{8\gamma_n})$. Note that $\gamma_n \to 0$, hence $r_n \to \infty$. Finally, denote $\Delta_n:=\{(x,y)\in X_n \times X_n: d_n(x,y)\leq r_n\}$. Then we have:
\begin{align*}
\mu_n(\Delta_n)&=\sum_{(x,y)\in \Delta_n} \mu_n(x,y) = \sum_{\substack{(x,y)\in Y_n \times Y_n\\d_n(x,y) \leq r_n}} m_n(x)m_n(y) = \sum_{x\in Y_n} m_n(x)\big( \sum_{\substack{y\in Y_n\\d_n(x,y) \leq r_n}} m_n(y)\big)\\
&\leq \sum_{x\in Y_n} m_n(x) \cdot K^{r_n} \cdot \gamma_n \leq m_n(Y_n) \cdot \frac{1}{8} \leq \frac{1}{8}.
\end{align*}
On the other hand, we have:
\[
\mu_n(X_n \times X_n) = \sum_{(u,v)\in Y_n \times Y_n} m_n(u)m_n(v) = m_n(Y_n)^2 \geq \frac{1}{4}.
\]
Hence we obtain:
\[
\mu_n\big( (X_n \times X_n) \setminus \Delta_n \big) \geq \frac{1}{4} - \frac{1}{8} = \frac{1}{8}.
\]

Now we consider another measure $\nu_n$ on $X_n \times X_n$ defined by
\begin{equation*}
\nu_n(x,y):=
\begin{cases}
  ~\frac{\mu_n(x,y)}{\mu_n( (X_n \times X_n) \setminus \Delta_n )}, & \mbox{if~} (x,y)\notin \Delta_n; \\
  ~0, & \mbox{otherwise.}
\end{cases}
\end{equation*}
It is clear that $\nu_n$ is a symmetric probability measure such that $\nu_n(x,y)=0$ if $d_n(x,y) \leq r_n$. Furthermore, for any measure space $(Z,\nu)$ and any map $f_n: X_n \to L^p(Z,\nu)$ satisfying $\|f_n(x)-f_n(y)\|_p \leq \rho_+(d_n(x,y))$, the composition $f_n\circ i_n:V_n \to \H$ satisfies that $\|f_n\circ i_n(x)-f_n\circ i_n(y)\|_p \leq \rho_+(Rd_n(x,y))$. Hence from Lemma \ref{lem:Cheeger to Poincare measured case}, there exists a constant $\kappa>0$ depending only on $K,s,c,p,R,\rho_+$ such that:
\begin{align*}
\sum_{x,y\in X_n}& \|f_n(x)-f_n(y)\|_p^p \nu_n(x,y) \leq 8\sum_{x,y \in X_n} \|f_n(x)-f_n(y)\|_p^p \mu_n(x,y)\\
& = 8\sum_{x,y\in Y_n} \|f_n(x)-f_n(y)\|_p^p m_n(x)m_n(y) \\
& =  8\sum_{u,v\in V_n} \|f_n\circ i_n(u)-f_n\circ i_n(v)\|_p^p m'_n(u)m'_n(v) \\
& \leq 8\kappa m'_n(V_n) = 8\kappa m_n(Y_n) \leq 8\kappa.
\end{align*}
This concludes the proof.
\end{proof}

Since measured expanders are measured asymptotic expanders, we obtain an immediate corollary:

\begin{cor}\label{cor:generalised expanders from random walk}
Let $\{(V_n,E_n,m_n)\}_{n \in \N}$ be a sequence of ghostly measured expanders with uniformly bounded valency.
Then $\{(V_n,d_n,m_n)\}_{n \in \N}$ is a sequence of $L^p$-generalised expanders for each $p\in [1,\infty)$ where $d_n$ is the edge-path metric on $V_n$.
\end{cor}

Combining with Proposition \ref{prop:generalised expanders}, we obtain the following corollary, which recovers \cite[Corollary 7.2]{measured_I}:

\begin{cor}\label{cor:non CE}
Let $\{(X_n,d_n,m_n)\}_{n\in \N}$ be a sequence of ghostly measured asymptotic expanders with uniformly bounded geometry.
Then the coarse disjoint union of $\{(X_n,d_n)\}$ does not coarsely embed into any $L^p$-space for $p\in [1,\infty)$.
\end{cor}

\section{Examples}

In this final section, we present some examples to explore the relation between classical expanders and the measured ones.

Recall that for a residually finite group with Kazhdan's property (T), its box spaces form sequences of expanders. This is the type of the first known explicit constructions of expanders due to Margulis \cite{MR0484767}. However, the following results show that on such expanders, one can always choose ghostly measures such that the associated measured graphs do not form a sequence of measured expanders.

The following Proposition was pointed out to us by G\'abor Elek:

\begin{prop}\label{prop:example 1}
Let $G$ be a residually finite group and $G=N_0 \rhd N_1 \rhd N_2 \rhd \ldots$ be a sequence of finite index normal subgroups with trivial intersection. Then for any $n\in \N$ there exists a probability measure $m_{n}$ on $G/N_n$ (without full support) such that the sequence $\{(G/N_n, m_n)\}_n$ is ghostly but not a sequence of measured expanders.
\end{prop}

\begin{proof}
  This follows from \cite{fraczyk2019heat}: the authors there show that heat kernels --- convolution powers of the probability measure describing a simple random walk --- on infinite, connected, bounded degree graphs are not (measured) expanders. Hence there is a sequence of finitely supported probability measures $\{m_n\}_{n\in\mathbb{N}}$ on $G$ as above, which is ghostly and non-expanding in measure. For each such measure $m_n$ on $G$ there exists $k_n\geq0$ such that for all $k\geq k_n$, the quotient map $G\to G/N_k$ is injective on 1-neighbourhood of $\supp(m_n)$. Thus we can inductively find a subsequence of the subgroups $(N_{l_n})_{n\in\mathbb{N}}$, such that we can endow $G/N_{l_n}$ with the (injective) push-forward of the measure $m_n$. By construction, the resulting sequence will be ghostly, but will not form a sequence of measured expanders.
\end{proof}

However, one can always perturb a sequence as in the above Proposition into a sequence where the measures are full measures:

\begin{prop}\label{prop:perturb-nonex-into-full-measures}
Let $\{(V_{n},E_{n},\mu_{n})\}_{n\in\mathbb{N}}$ be a sequence of ghostly measured graphs, which is \emph{not} a sequence of measured expanders, and such that the probability measures $\mu_{n}$, $n\in\mathbb{N}$, do not have full support. Then there exists a sequence of ghostly probability measures $\{\mu'_{n}\}_{n\in\mathbb{N}}$ on $\{(V_{n},E_{n})\}_{n\in\mathbb{N}}$ with full support, such that $\{(V_{n},E_{n},\mu'_{n})\}_{n\in\mathbb{N}}$ is \emph{not} a sequence of measured expanders.
\end{prop}

\begin{proof}
  After taking a subsequence if necessary, we can find $A_n \subseteq V_{n}$ with $0< \mu_n(A_n) \leq \frac{1}{2}$ such that $\frac{\mu_n(\partial A_n)}{\mu_n(A_n)}$ tends to $0$ as $n \to \infty$. We construct a new measure $\mu'_n$ on $V_{n}$ as follows:
\[
	\mu'_n(\{x\})=
	\begin{cases}
	~(1-\frac{\mu_n(A_n)}{n}) \mu_n(\{x\}), & \mbox{if~} x\in \supp(\mu_n);  \\[0.3cm]
	~\frac{\mu_n(A_n)}{n} \cdot \frac{1}{|(G/N_{l_n}) \setminus \supp(\mu_n)|}, & \mbox{if~} x\in V_{n} \setminus \supp(\mu_n).
	\end{cases}
	\] 
Since $0< \mu_n(A_n) \leq \frac{1}{2}$, it is clear that $\mu_n'$ has full support and the sequence $\{(V_{n}, \mu'_n)\}_n$ is ghostly. Finally, we observe:
\[
\frac{\mu_n'(\partial A_n)}{\mu'_n(A_n)}  \leq \frac{\mu_n(\partial A_n) + \frac{\mu_n(A_n)}{n}}{(1-\frac{\mu_n(A_n)}{n}) \mu_n(A_n)} \leq \frac{2\mu_n(\partial A_n)}{\mu_n(A_n)} + \frac{1}{n-\mu_n(A_n)} \to 0
\]
as $n\to \infty$. Hence we obtain that $\{(V_{n}, \mu'_n)\}_n$ does not form a sequence of measured expanders, which concludes the proof.
\end{proof}

As a direct corollary (by taking the group $G$ to have property (T) in Proposition \ref{prop:example 1}), we obtain the following:
\begin{cor}\label{exm1}
There exist a sequence of expanders $\{(V_n,E_n)\}_{n\in \N}$ and measures $m_n$ on $V_n$ with full support such that $\{(V_n, E_n, m_n)\}_{n\in \N}$ is ghostly but does not form a sequence of measured expanders.
\end{cor}

Proposition \ref{prop:example 1} shows that property (T) cannot ensure that box spaces with arbitrary chosen measure form sequences of measured expanders. Hence it is natural to ask:

\begin{question}
Can we define a notion of \emph{measured property (T)} for groups to ensure that box spaces with any ghostly measures form sequences of measured expanders?
\end{question}

On the other hand, recall that there is a weaker condition for groups called property $(\tau)$ introduced in \cite{lubotzky2005property} which also ensures that quotients by a fixed sequence of finite index normal subgroups with trivial intersection form a sequence of expanders. Hence one may also explore a measured version of property $(\tau)$ to construct explicit examples of measured expanders.

The next example concerns a permanence property of measured expanders. First note the following elementary observation, whose proof is straightforward, hence omitted.

\begin{lem}
Let $\{(V_n,E_n,m_n)\}_n$ be a sequence of finite probability measured graphs, and $m'_n$ be another sequence of probability measures on $V_n$. Assume that there exists a constant $s\geq 1$ such that for each $n\in \N$ and $x\in V_n$, we have $\frac{1}{s}m_n(x) \leq m'_n(x) \leq sm_n(x)$. Then $\{(V_n,E_n,m_n)\}_n$ is a sequence of measured expanders \emph{if and only if} $\{(V_n,E_n,m'_n)\}_n$ is a sequence of measured expanders.
\end{lem}

However, without the assumption on measure ratio, such a conclusion is false. The following result was communicated to us by G\'abor Elek. Note that the notion of Cheeger constant can be defined for an arbitrary (not necessarily finite) measured graph $(V,E,m)$ in exactly the same way as (\ref{EQ:measured cheeger}) provided that $m(V)< \infty$.


\def\gr{\mathcal{G}}

\begin{lem}\label{lem:example 2}
Let $\gr=(V,E,\mu)$ be a measured graph with $\mu(V)=1$ and Cheeger constant $c=ch(V,E,\mu)>0$.
  Let $\gr_\infty = (V_\infty,E_\infty) = \gr\times \mathbb{N}_0$ be the Cartesian product of graphs, where we consider $\mathbb{N}_0$ to be a graph with edges between consecutive numbers.  Endow $\gr_\infty$ with the measure $\mu_\infty$, defined by declaring that $\mu_\infty|_{V\times\{i\}} = 2^{-i}\mu$ for $i\in\mathbb{N}_0$. Then the Cheeger constant of $\gr_\infty=(V_\infty, E_\infty, \mu_\infty)$ is at least $\min(\frac{c}{18},\frac{1}{8})$.
\end{lem}

\begin{proof}
Consider a non-empty subset $A\subset V_\infty$ with $\mu_\infty(A)\leq 1 = 	\mu_\infty(V_\infty)/2$. For $i\in\mathbb{N}_0$, denote $\gr_i=\gr \times \{i\}$, $V_{i} = V\times\{i\}$, $A_i = A\cap V_i$, $A_{[i,\infty)} = \sqcup_{j=i}^\infty A_j$ and $\mu_i:=\mu_\infty|_{V\times\{i\}}$.
  
  If $\mu_i(A_i)\leq \mu_i(V_i)/2$ for all $i\in\mathbb{N}_0$, then, just using the expansion assumption in each $\gr_i$ (which is a scaled copy of $G$) separately, we obtain $\mu_\infty(\partial A)\geq c\mu_\infty(A)$.
  
  If the above is not the case, then let $k\in\mathbb{N}_0$ be the smallest number for which $\mu_k(A_k)> \mu_k(V_k)/2$. Observe that the slices of $A$ after index $k$ can't be too big: \[
  \mu_\infty\left(A_{[k+1,\infty)}\right) \leq \mu_\infty\left(V_{[k+1,\infty)}\right) = \sum\nolimits_{i=k+1}^\infty 2^{-i} = 2^{-k} = \mu_k(V_k)\leq 2\mu_k(A_k).
  \]

  Case 1: $k>0$ and $\mu_{k-1}(A_{k-1})< \mu_{k-1}(V_{k-1})/4$. Thus also $\mu_{k-1}(A_{k-1})< \mu_k(V_k)/2 < \mu_k(A_k)$, and hence
  \[\mu_\infty\left(A_{[k-1,\infty)}\right) = \mu_{k-1}(A_{k-1})+\mu_k(A_k)+\mu_\infty\left(A_{[k+1,\infty)}\right)\leq 4\mu_k(A_k).
  \]
  In this case we examine the boundary of $A$ within the slice $\gr_{k-1}$. It contains at least ``a copy'' of $A_k$ (in $\gr_{k-1}$), but if they lie within $A_{k-1}$, they are not boundary points. Hence we obtain:
  \begin{equation*}
  \mu_\infty(\partial(A_{[k-1,\infty)})\cap V_{k-1}) \geq 2\mu_k(A_k) - \mu_{k-1}(A_{k-1})> \mu_k(A_k).
  \end{equation*}
  We split $A = \left(\sqcup_{i=0}^{k-2}A_i\right) \sqcup A_{[k-1,\infty)}$. The first part (if not empty) expands at least by a factor of $c$ (in each slice separately). For the other part, combining the above inequalities, we obtain
  \[
  \mu_\infty\left(\partial(A_{[k-1,\infty)})\cap V_{[k-1,\infty)}\right)
  \geq \mu_k(A_k) \geq \mu_\infty\left(A_{[k-1,\infty)}\right)/4.
  \]
  Hence this part expands at least by a factor of $\frac{1}{4}$, which finishes Case 1.
  
  Case 2: $k>0$ and $\mu_{k-1}(A_{k-1})\geq \mu_{k-1}(V_{k-1})/4$. Here we can simply include $A_k$ into the tail since
  $\mu_\infty\left(A_{[k,\infty)}\right) \leq 2^{-k+1} =\mu_{k-1}(V_{k-1}) \leq 4 \mu_{k-1}(A_{k-1})$, and thus $\mu_\infty\left(A_{[k-1,\infty)}\right) \leq 5\mu_{k-1}(A_{k-1})$. Decomposing $A$ as in Case 1, the first part again expands at least by a factor of $c$ in each slice; and for the second part we estimate
  \begin{align*}
  \mu_\infty\left(\partial(A_{[k-1,\infty)})\cap V_{[k-1,\infty)}\right)
  &\geq \mu_{k-1}\left(\partial(A_{k-1})\cap V_{k-1}\right)
  \geq c\mu_{k-1}\left(A_{k-1}\right)\\
  &\geq \frac{c}{5}\mu_\infty\left(A_{[k-1,\infty)}\right).
  \end{align*}
  This finishes Case 2.
  
  Case 3: $k=0$ and $\mu_0(A_0)\geq \frac{3}{4}$. Consequently $\mu_1(A_1)\leq \mu_\infty\left(A_{[1,\infty)}\right)\leq\frac{1}{4}\leq\mu_0(A_0)/3$, and $\mu_\infty(A)\leq 1\leq \frac{4}{3}\mu_0(A_0)$.
  Thus using just the $\gr_1$-part of the boundary of $A$, we have
  \[
  \mu_\infty(\partial(A))\geq \mu_1(\partial(A)\cap V_1) \geq \mu_0(A_0)/2-\mu_1(A_1)\geq \mu_0(A_0)/6 \geq \mu_\infty(A)/8.
  \]
  
  Case 4: $k=0$ and $\mu_0(A_0)<\frac{3}{4}$. Using \cite[Lemma 4.12]{measured_I} and \cite[Lemma 6.3]{dypartI} we have that $\mu_0(\partial(A_0)\cap V_0) \geq \frac{c}{6}\mu_0(A_0)$. Recall from the beginning of the proof that $\mu_\infty\left(A_{[1,\infty)}\right) \leq 2\mu_0(A_0)$. Hence combining them together, we obtain
  \[
  \mu_\infty(\partial(A)) \geq \mu_0(\partial(A_0)\cap V_0)
  \geq \tfrac{c}{6}\mu_0(A_0)
  \geq \tfrac{c}{18}\mu_\infty(A).
  \]
  This finishes Case 4, and the whole proof.
\end{proof}

\begin{ex}
Let $\{\mathcal{H}_n=(V_n,E_n,m_n)\}_{n\in\mathbb{N}}$ be a sequence of (classical) expanders, endowed with the normalised counting measures. For every $n\in\mathbb{N}$ define $\mathcal{K}_n = \mathcal{H}_n \times \{0,1,2,\ldots,n\}$ where we consider $\{0,1,2,\ldots,n\}$ as a subgraph in $\mathbb{N}_0$. By a standard argument, $\{\mathcal{K}_n\}_{n\in\mathbb{N}}$ is \emph{not} a classical expander sequence (when no measures are considered). However, we can endow each $\mathcal{K}_{n}$ with the restriction of the measure $(m_{n})_{\infty}$ described in Lemma \ref{lem:example 2}. Using the conclusion of the Lemma and an approximation argument, $\{\mathcal{K}_n\}_{n\in\mathbb{N}}$ is a sequence of measured expanders.
\end{ex}

  In conclusion, we obtain the following:
  
  \begin{prop}\label{exm2}
  There exists a sequence of ghostly measured expanders $\{(V_n, E_n, m_n)\}_{n\in \N}$ with uniformly bounded valency such that the underlying graphs $\{(V_n, E_n)\}_{n\in \N}$ are not expanders.
  \end{prop}

\bibliographystyle{plain}
\bibliography{expanderish_graph_measured}

\end{document}